\documentclass[11pt]{article}

\usepackage[a4paper,margin=1.4in]{geometry}
\usepackage{amsmath,amssymb,amsfonts,amsthm}
\usepackage{amsmath,amssymb,amsthm}
\usepackage{amsmath, amssymb}
\usepackage{pgfplots}
\pgfplotsset{compat=1.18}
\usetikzlibrary{patterns}

\usepackage{float}

\usepackage{tikz}
\usetikzlibrary{positioning}
\usepackage{xcolor}
\usepackage{tikz-cd}
\usepackage[ruled,vlined,linesnumbered]{algorithm2e}

\usepackage{amscd}   
\usepackage{tikz-cd} 
\usepackage{pgfplots}
\pgfplotsset{compat=1.18}
\usepackage[colorlinks=true,
            citecolor=blue,
            linkcolor=blue,
            urlcolor=blue]{hyperref}

\newtheorem{theorem}{Theorem}[section]
\newtheorem{lemma}[theorem]{Lemma}
\newtheorem{proposition}[theorem]{Proposition}
\newtheorem{definition}[theorem]{Definition}
\newtheorem{remark}[theorem]{Remark}
\newtheorem{example}[theorem]{Example}
\newtheorem{Corollary}[theorem]{Corollary}
\usepackage{xcolor}
\usepackage{hyperref}
\hypersetup{
    colorlinks=false,
    pdfborder={0 0 1.2},
    linkbordercolor={red},     
    citebordercolor={green},   
    urlbordercolor={blue}
}

\title{\textbf{On Projective and Flat Persistence Modules}}

\author{
Prateep Chakraborty\thanks{
Department of Mathematics;
Indian Institute of Technology Kharagpur;
Kharagpur, West Bengal,  721302, India. \quad
\textit{Email:} prateep@maths.iitkgp.ac.in}
\and
Giriraj Ghosh\thanks{
Department of Mathematics;
Indian Institution of Technology Kharagpur;
Kharagpur, West Bengal,  721302, India. \quad
\textit{Email:} gghoshmaths@kgpian.iitkgp.ac.in}
}

\date{}
\begin{document}
\maketitle
\thispagestyle{plain}

\begin{abstract}
 In recent years, persistence modules have been viewed as graded modules with gradation over a preordered set serving as the indexing set. We provide sufficient criteria for a projective module over a PID to be free when the indexing set is a lattice. With a lattice as the indexing set, we obtain criteria ensuring that a given persistence module is not projective. When the indexing set is a preordered set, we establish the flatness of a well-known family of persistence modules. We end the article with two algorithms to compute a basis of free persistence modules with indexing sets $\mathbb{Z}$ and $\mathbb{Z}^2$.
\end{abstract}

\vspace{0.2cm}


\noindent
\textbf{22020 Mathematics Subject Classification} \qquad Primary: 55N31, 13C10, 13C11;  \qquad  Secondary: 06A07, 06A11.
\vspace{0.8cm}







\section{Introduction}

Persistence modules arise naturally in topological data analysis as algebraic invariants associated with filtered topological spaces. 
\begin{definition}
A \textbf{persistence module} \(M\) over a ring \( R \) is a sequence of \( R \)-modules
\[
M_1 \xrightarrow{f_{1,2}} M_2 \xrightarrow{f_{2,3}} M_3 \xrightarrow{f_{3,4}} \cdots \xrightarrow{f_{n-1,n}} M_n
\]
together with \( R \)-module homomorphisms among them such that for each \( i \leq j \),
the map \( f_{i,j} : M_i \to M_j \) is defined as the composition
\(
f_{i,j} = f_{j-1,j} \circ f_{j-2,j-1} \circ \cdots \circ f_{i,i+1},
\)
and \( f_{i,i} \) is the identity map on \( M_i \).
The sequence describes how the modules \emph{persist} through the indices \( i \) to \( j \)
for certain intervals, depending on the context.
\end{definition}
When the indexing set is a finite subset of $\mathbb{R}$ (\cite{zomorodian2005computing}, \cite{gabriel1972unzerlegbare}
) and the persistence module satisfies the pointwise finite–dimensional condition (\cite{webb1985decomposition}, \cite{crawleyboevey2015decomposition}) or $q$-tame condition (\cite{chazal2016structure}), the persistence module decomposes uniquely (up to isomorphism) into interval modules. The theory has since been extended to more general indexing sets such as $\mathbb{R}^n$, $\mathbb{Z}^n$, and partially ordered sets. In \cite{bubenik2021homological}, Bubenik and Milićević further generalizes the framework by introducing persistence modules indexed by a preordered set.

From an algebraic viewpoint, persistence modules have been studied as graded modules (\cite{zomorodian2005computing}, \cite{carlsson2009multidimensional}, \cite{lesnick2015interleaving}),
 as functors (\cite{bubenik2014categorification}, \cite{bubenik2015metrics}), and as sheaves (\cite{curry2014sheaves}). Bubenik and Milićević developed several aspects of the graded–module and sheaf–theoretic approaches. Using these two perspectives, they defined tensor products, $\mathrm{Hom}$ bifunctors, and the derived functors $\mathrm{Tor}$ and $\mathrm{Ext}$ for persistence modules. Within this setup they obtained a criterion for projectivity: 
 
 \textit{With a preordered set $P$ together with an abelian group structure as the indexing set,  a finitely generated persistence module over a field is graded projective if and only if it is graded free (see \cite[Corollary~2.20]{bubenik2021homological}).}
 
 In particular, with the indexing set $\mathbb{R}$, they also characterized when interval modules are projective, flat, or injective. The projective persistence modules are discussed in \cite{blanchette2023exact} too.

Since Bubenik and Milićević's results apply to arbitrary preordered sets, they are in particular valid for lattices. On a lattice, we generalize this result as follows. 
\begin{theorem}\label{0}
Let \( R \) be a principal ideal domain, and let \( P \) be a lattice equipped with a compatible abelian group structure. Suppose with the indexing set $P$, \( M \) is a persistence module over $R$ such that 
\[
M = \bigoplus_{a \in P} M_a,
\] and each $M_a$ is finitely generated $R$-module. Then \( M \) is graded projective if and only if it is graded free.

(See Section \ref{sec2} for a detailed discussion of the necessary concepts.)
\end{theorem}

Note that our results hold over a PID rather than only over a field. When $R$ is a field, the condition that each $M_a$ is a finitely generated $R$-module coincides with the notion of $q$-finite modules introduced in \cite{miller2020posets}. Moreover, for $P=\mathbb{Z}^n$ and $R=\mathbb{K}$ a field, our result yields a proof of the graded analogue of Serre’s conjecture (See \cite{lam1978serre}, \cite{serre1955fac}).

The proof proceeds by introducing a notion of `indecomposable elements' adapted to
the graded persistence setting. By applying the notion of a minimal element of certain subsets of $P$, we show the existence of indecomposable elements. By analyzing the degrees at which indecomposable elements appear
and by controlling their linear independence, we construct a homogeneous basis and
establish freeness.  This approach is intrinsically homological and does not rely
on interval decompositions.

In the process of proving Theorem \ref{0}, we show that any homogeneous element of a projective persistence module can be written `in terms of indecomposable elements' (see Lemma \ref{propp}), which is applied in Section \ref{sec5} to derive criteria ensuring that certain persistence modules cannot be projective (see Lemma \ref{lemfalt}). In \cite{bubenik2021homological}, the authors provide a collection of graded flat modules where the indexing set is $\mathbb{R}$. Generalizing the indexing set to any preordered set with a compatible abelian group structure, a collection of graded flat modules is constructed in this article. In conclusion, we construct broad families of flat persistence modules that fail to be projective (see Example \ref{exx}). These constructions demonstrate that flatness and projectivity diverge in meaningful ways over a PID, especially in lattice-indexed settings.

Our main results provide a method for constructing bases of projective (that is free) persistence modules. We conclude by giving an explicit algorithm to compute such bases in the cases $P=\mathbb{Z}$ and $P=\mathbb{Z}^2$.

The paper is organized as follows. In Section~\ref{sec2} we recall the necessary background and the interpretation of a persistence module as a graded module over certain monoid ring. In Section \ref{sec3} we introduce the notion of \emph{'indecomposable elements`}, establish their existence, and prove auxiliary results used later. Section \ref{sec4} shows that a linearly independent set of indecomposable elements forms a basis of a projective persistence module. In Section \ref{sec5} we construct a family of flat modules that are not projective. Finally, Section \ref{sec6} presents an algorithm for producing a basis of a free persistence module.

\section{Preliminaries}\label{sec2}
 In this section, we recall the notions of preordered sets, posets, group rings, graded rings, graded modules, graded projective modules, and graded flat modules.
We also define the notion of persistence modules over preordered sets with a compatible abelian group structure and view it as graded modules over a certain monoid ring, following the approach of Bubenik and Milićević~\cite{bubenik2021homological}.

\subsection{Preordered sets and posets}
We refer the reader to \cite{bubenik2021homological} for a detailed discussion of this subsection. Given a preordered set $(P,\leq)$, we can associate a category $\mathbf{P}$ whose objects are the elements of $P$ and whose morphisms are given by the relations $x \leq y$.
With $P$ as the indexing set, a persistence module over $R$ is then defined as a functor
\[
M \colon \mathbf{P} \to \mathbf{A},
\]
where $\mathbf{A}$ denotes the category of modules over a principal ideal domain $R$.
The morphisms between persistence modules are natural transformations between such functors.
\begin{definition}
Let $(P,\le)$ be a preordered set. A subset $U$ of $P$ is called an up-set if for any $x \in U$ and $y \in P$ with $y \ge x$, we have $y \in U$. Similarly, a subset $V$ of $P$ is called a down-set of $P$ if for any $x \in V$ and $y \in P$ with $y \le x$, we have $y \in V$.

For any $a \in P$, the principal up-set $U_a$ is defined as
\[
U_a = \{ x \in P \mid x \ge a \}.
\]
Similarly, the principal down-set $V_a$ is defined as
\[
V_a = \{ x \in P \mid x \le a \}.
\]

    A subset $S$ of $P$ is called convex if for $a \le c \le b$ with $a, b \in S$ and $c \in P$, we have $c \in S$. Note that both up-sets and down-sets are convex.

\end{definition}
\begin{definition}
   We now define a \textbf{minimal} element of a subset \( S \subseteq P \), if it exists, as an element \( m \in S \) such that there is no \( s \in S \) with \( s < m \); that is, \( s \leq m \) implies \( s = m \).  
Similarly, a \textbf{maximal} element of \( S \) is an element \( m' \in S \) such that there is no \( s \in S \) with \( s > m' \). 
\end{definition}

Recall that a partially ordered set (or poset) is a particular kind of preordered set.
For a poset $(P,\leq)$ and elements $x,y \in P$, we define the \textbf{join} (or supremum) of $x$ and $y$, denoted by
\(
x \vee y,
\)
to be the smallest element $r \in P$ such that $x \leq r$ and $y \leq r$.
Dually, we define the \textbf{meet} (or infimum) of $x$ and $y$, denoted by
\(
x \wedge y,
\)
to be the largest element $s \in P$ such that $s \leq x$ and $s \leq y$.
Thus, for any $p,q \in P$, the following implications hold:
\(
x \leq p \text{ and } y \leq p \;\Rightarrow\; x \vee y \leq p,
\)
and
\(
q \leq x \text{ and } q \leq y \;\Rightarrow\; q \leq x \wedge y.
\)
By antisymmetry of the partial order, the join and meet of two elements $x,y \in P$ are unique whenever they exist.
A poset in which every pair of elements admits both a join and a meet is called a \textbf{lattice}.
Observe that in a lattice, every finite subset has both a join and a meet.

A preordered set $(P,\leq,+,0)$ is said to have a \textbf{compatible abelian group structure} if $(P,+,0)$ is an abelian group equipped with a preorder $\leq$ such that
\[
a \leq b \;\Longrightarrow\; a + c \leq b + c
\quad \text{for all } a,b,c \in P .
\]
\begin{example}
Let \( P = \mathbb{R}^n \) with the \textbf{product (coordinatewise) partial order}
defined by
\(
x = (x_1, \dots, x_n) \leq y = (y_1, \dots, y_n)
\; \text{if and only if} \;
x_i \leq y_i \text{ for all } i = 1, \dots, n.
\)
Then \( (\mathbb{R}^n, \leq) \) is a lattice, where the join and meet are given coordinatewise by
\(
(x \vee y)_i = \max(x_i, y_i), \;
(x \wedge y)_i = \min(x_i, y_i).
\)

\end{example}
\begin{example}
Let
\(
P = \mathbb{Z}^{i_1} \times \mathbb{Q}^{i_2} \times \mathbb{R}^{i_3},
\)
where $i_1,i_2,i_3 \ge 0$ and $i_1 + i_2 + i_3 \ge 1$.
Equipped with the product order and componentwise addition, $P$ is a preordered set with a compatible abelian group structure.
\end{example}

\subsection{Graded projective and graded flat modules}

We refer the reader to \cite{nastasescu2004graded} for the theory of graded rings and graded modules.
For a concise treatment of graded modules in the context of persistence, see Appendix~B of \cite{bubenik2021homological}.
In particular, these references cover graded free modules, graded projective modules, and graded flat modules.
Below, we recall the definitions of graded projective and graded flat modules.
\begin{definition}\label{def}
Let $G$ be a group and \( R = \bigoplus_{g \in G} R_g \) be a graded ring. 
A graded \( R \)-module \( P = \bigoplus_{g \in G} P_g \) is said to be \textbf{graded projective} 
if it satisfies any of the following equivalent conditions:

\begin{enumerate}
   \item For every graded epimorphism \( f : M \twoheadrightarrow N \), a graded homomorphism 
    \( g : P \to N \) can be \textbf{lifted} through \( f \); that is, there exists a graded homomorphism 
    \( h : P \to M \) such that \( f \circ h = g \).
   \item The functor 
    \(
    \underline{\mathrm{Hom}}_R(P,-)
    \)
    from the category of graded \( R \)-modules to itself is \textbf{exact}.
    \item There exists a graded free \( R \)-module \( F \) such that 
    \( P \) is a \textbf{graded direct summand} of \( F \); that is,
    there exists a graded submodule \( Q \subseteq F \) with 
    \(
    F = P \oplus Q.
    \)
\end{enumerate}
\end{definition}
\begin{definition}\cite{bubenik2021homological}
    A graded module $M$ is $\otimes_{\mathrm{gr}}$-flat if $-\otimes_{\mathrm{gr}} M$ is an exact functor.

\end{definition}

\subsection{Persistence modules as graded modules}\label{sub23}
We now recall how persistence modules can be viewed as graded modules.
See \cite{bubenik2021homological} for a detailed discussion of this correspondence.
\paragraph{Monoid ring:}
Let $(P,\leq,+,0)$ be a preordered set with a compatible abelian group structure, and let $U_0 \subseteq P$ denote the upset of nonnegative elements.
Then $(U_0,+,0)$ is a commutative monoid.
For a principal ideal domain $R$, let $R[U_0]$ denote the corresponding monoid ring (see~\cite{bubenik2021homological}); this construction is analogous to that of a group ring.

An arbitrary element of $R[U_0]$ can be written as a finite formal sum
\(
\sum_{s \in U_0} c_s t^s,
\)
where $c_s \in R$ and only finitely many coefficients $c_s$ are nonzero.
Since both $R$ and $U_0$ are commutative, the ring $R[U_0]$ is commutative.

A grading on $R[U_0]$ is defined by declaring $\deg(c_s t^s)=s$.
Thus,
\[
R[U_0] = \bigoplus_{s \in P} (R[U_0])_s,
\]
where $(R[U_0])_s \cong R$ if $s \ge 0$ and $(R[U_0])_s = 0$ otherwise.

\paragraph{Graded module structure of a persistence module:}

Let $M \colon \mathbf{P} \to \mathbf{A}$ be a persistence module, viewed as a functor as defined earlier.
We define a left $R[U_0]$--action on $M$ as follows.
For $m \in M_i$ and $s \in U_0$, set
\[
t^s \cdot m := M_{i \le i+s}(m),
\]
where $t^s$ denotes the generator of $R[U_0]$ in degree $s$.
This action extends $R$--linearly and endows $M$ with the structure of a graded left $R[U_0]$--module. Further the same action induces right $R[U_0]$-module structure on $M$.
Conversely, given a graded left $R[U_0]$--module structure on $M$, multiplication by $t^s$ induces natural maps
\(
M_i \xrightarrow{\; t^s \;} M_{i+s},
\)
which define a persistence module structure on $M$.

Moreover, a natural transformation between persistence modules (viewed as functors) corresponds precisely to a graded $R[U_0]$--module homomorphism.
Consequently, this correspondence induces an equivalence of categories between the category of functors
$\mathbf{P} \to \mathbf{A}$ and the category of graded $R[U_0]$--modules. This correspondence has been noted by several authors: in the case of $P=\mathbb{R}^n$ in \cite{lesnick2019interactive}, for $P=\mathbb{Z}^n$ in \cite{zomorodian2005computing}, for partially ordered abelian groups in \cite{miller2020posets}, and for general preordered sets in \cite{bubenik2021homological}.
Bubenik and Milićević proved that, under suitable conditions, a graded $R[U_0]$--module $M$ is free whenever it is projective; see Corollary~2.20 of~\cite{bubenik2021homological}.
We conclude this subsection by providing an alternative definition of the tensor product of two persistence modules, viewed as $P$--graded $R[U_0]$-modules.

\textbf{Alternate definition:  }
    Let $M$ and $N$ be $P$-graded $R[U_0]$-modules. The $P$-graded abelian group $M \otimes_{\mathrm{gr}} N$ by setting
\[
(M \otimes_{\mathrm{gr}} N)_r := \operatorname{colim}_{s+t \le r} (M_s \otimes_R N_t).
\]

\textit{\textbf{For Section \ref{sec3} and Section \ref{sec4} we assume $M$ is a persistence module with the indexing set $P$, where $P$ is a lattice with a compatible abelian group structure. For Section ~\ref{sec3} and Section ~\ref{sec4} (excluding Subsection \ref{sec7}), we assume that $M$ is graded projective with $M_a$, for $a\in P$, is finitely generated $R$-module.}} With this hypothesis there exists a graded free $R[U_0]$-module $F$ and a graded $R[U_0]$-module $N$ such that
\(
F \cong M \oplus N .
\)
Since $F$ is a graded free $R[U_0]$-module, let $\{u_\eta\}_{\eta\in\Lambda}$ be a basis of homogeneous elements of $F$.
\section{Some necessary results}\label{sec3}
We define a natural concept of \emph{`indecomposable element'} in this section. 
A set of indecomposable elements yields a basis for the graded $R[U_0]$-module $M$ (see Proposition \ref{15}), 
which proves that $M$ is a free module. 
In this section, we show the existence of indecomposable elements (see Corollary~\ref{9}) 
and obtain the finiteness of the cardinality of linearly independent 
indecomposable elements of degrees below a certain degree (see Proposition~\ref{14}).

As a \(P\)-graded free \(R[U_0]\)-module \(F\), we have
\[
F_i = M_i \oplus N_i \quad \text{for all } i\in P.
\]
Since each \(M_i\) is a finitely generated projective \(R\)-module,
it follows that \(M_i\) is a free \(R\)-module with finitely many
basis elements, applying the structure theorem for finitely generated modules over a PID.

Thus, for some \(i\in P\) with \(M_i\neq 0\), let
\(
x_1^i,x_2^i,\dots,x_n^i
\)
be an \(R\)-basis of \(M_i\).
\newpage
\begin{lemma}\label{7}
For $k \leq i$, the map $f_{k,i}:M_k \to M_i$ is injective.
\end{lemma}
\begin{proof}
    Let \( z \in F_k \) and write \( z = \sum_{\eta \in I_z} c_\eta u_\eta \), $c_\eta \in R[U_0]$. Here \(I_z\) denotes an index set, which is a subset of \(\Lambda\),
consisting of all indices \(\eta\) such that \(c_\eta\neq 0\) in the
representation of \(z\) as an \(R[U_0]\)-linear combination of the basis \(\{u_\eta\}\).
  
If \( f_{k,i}(z) = 0 \), then
\(
\sum c_\eta t^{\,i-k} u_\eta = 0,
\)
which implies each \( c_\eta t^{\,i-k} = 0 \).  
Writing each \( c_\eta = \sum r_p t^p \) as an element of \( R[U_0] \), we get
\(
\sum r_p t^{\,p + (i-k)} = 0.
\)
Since $t^p$'s are distinct, so $t^{p+(i-k)}$'s are also distinct. It follows from the above equation that \( r_p = 0 \) for all \( p \), hence \( c_\eta = 0 \) for all \( \eta \).  
Thus \( z = 0 \), and so the kernel of \( f_{k,i} \) is trivial.  
Therefore, the map \( F_k \to F_i \) is one-to-one.

 Let \( w \) be any element in \( M_k \) such that \( f_{k,i}|_{M_k}(w) = 0 \).  
Since \( w \in F_k \) and \( F_k = M_k \oplus N_k \) as graded, we can write 
\[
f_{k,i}(w) = f_{k,i}(w + 0) = f_{k,i}|_{M_k}(w) + f_{k,i}|_{N_k}(0) = 0 + 0 = 0.
\]
But \( F_k \to F_i \) is one-to-one, hence \( w = 0 \).  
This implies that the kernel of \( f_{k,i}|_{M_k} \) is also trivial, so the map \( M_k \to M_i \) is injective.

\end{proof}
Similar to the notion of the birth of an element in a persistence module over $R$ 
(see Section~3.2, \cite{dey2022computational}), we define the birth set of an element $x$ 
in $M$ and in $F$ separately.

\begin{definition}
    Let $x \in M_i$. The \emph{birth set} of $x$ is defined as
\[
S^M_x = \{ \ell \leq i : M_\ell \ \text{has a preimage of } x\}.
\]
Similarly, one defines
\[
S^F_x = \{ \ell \leq i : F_\ell \ \text{has a preimage of } x\}.
\]
\end{definition}

Note that \( S_x^M \) and \( S_x^F \) are both nonempty subsets of \( P \), since \( i \) belongs to each of them. First, we show that these two sets are equal.

\begin{proposition}\label{1}
The two subsets of $P$, $S_x^M$ and $S_x^F$, are equal; that is, $S_x^M = S_x^F$.

\end{proposition}

\begin{proof}
Let \( l \in S_x^M \). Then there exists \( y \in M_l \) such that \( f_{l,i}(y) = x \).  
Since \( F_a = M_a \oplus N_a \) for all \( a \in P \), we have \( M_l \subseteq F_l \), and under the same map, \( y \) goes to \( x \) from \( F_l \) to \( F_i \).  
Hence \( l \in S_x^F \), that is, \( S_x^M \subseteq S_x^F \). \hfill (1)

On the other hand, let \( l' \in S_x^F \). Then there exists \( y' \in F_{l'} \) such that \( f_{l',i}(y') = x \).  
We can write \( y' = z' + w' \) with \( z' \in M_{l'} \) and \( w' \in N_{l'} \).  
Now \( f_{l',i}(z' + w') = x \) implies \( f_{l',i}(w') = x - f_{l',i}(z') \), which belongs to \( N_i \) (from the LHS) and to \( M_i \) (from the RHS).  
This means \( f_{l',i}(w') \in M_i \cap N_i = \{0\} \).  
Hence \( x - f_{l',i}(z') = 0 \), and we obtain an element \( z' \in M_{l'} \) mapping to \( x \) in \( M_i \) under the same map.  
Therefore \( l' \in S_x^M \), i.e., \( S_x^F \subseteq S_x^M \). \hfill (2)

From (1) and (2), it follows that \( S_x^M = S_x^F \).

\end{proof}

Since \( 0 \neq x \in M_i \subseteq M \subseteq F \), we can write
\[
x = \sum_{\eta \in I_x} c_\eta t^{i-|u_\eta|} u_\eta,
\]
where $I_x$ is a subset of $\Lambda$, such that \( c_\eta \neq0 \;(\in R) \) for all $\eta\in I_x$ and $|u_\eta|$ is degree of $u_\eta$.
Let \(m\) be the join of $\{\deg(u_\eta):\eta\in I_x\}$.

We note down the following obvious observation.
\begin{remark}\label{2}
For all $\eta \in I_x$, $|x|\ge |u_\eta|.$ 
\end{remark}
We are now ready to obtain the minimal element (it will be unique) of $S_x^F=S_x^M$. First we prove the following result about $S_x^F.$

\begin{proposition}\label{3}
    For any \(l\in P\), \(l\in S_x^F, \text{ implies } i \ge l\ge m\).

\end{proposition}

\begin{proof}
    Let \(l\in S_x^F\). Then there exists \(y\in F_l\) with \(f_{l,i}(y)=x\).
Write \(
y=\sum_{\eta\in I_y}d_\eta t^{l-|u_\eta|}u_\eta\), \(d_\eta \in R.
\)
By the Remark \ref{2} applied to \(y\) we have \(\deg(u_\eta)\le \deg(y)=l\) for all \(\eta\in I_y\).
Also
\(
t^{\,i-l}\sum_{\eta\in I_y}  d_\eta t^{l-|u_\eta|}u_\eta \;=\; \sum_{\eta\in I_x} c_\eta t^{i-|u_\eta|} u_\eta.
\)
Equality of the two sums implies \(I_x=I_y\), which in turn implies that
\(\deg(u_\eta)\le l\) for all \(\eta\in I_x\). Hence \(m=\text{join of }\{\deg(u_\eta):\eta\in I_x\}\), which is \(\le l\).

\end{proof}
We now obtain the minimal element of the set $S_x^F.$
\begin{lemma}\label{4}
    The minimal element of $S_x^F$ is $m.$
\end{lemma}
    \begin{proof}
       Let $x\in M_i.$ Since \(i\ge m\ge |u_\eta| \;\forall \eta\), we can rewrite
\[
x=\sum_{\eta\in I_x} c_\eta t^{i-|u_\eta|} u_\eta
   = \sum_{\eta\in I_x} c_\eta t^{i-m+m-|u_\eta|} u_\eta
   = t^{i-m}\cdot\sum_{\eta\in I_x} c_\eta t^{\,m-|u_\eta|} u_\eta,
\]
so \(m\in S_x^F\).

We will prove \(m\) is the minimum of \(S_x^F\). If there exists some \(m'\in S_x^F\) (with \(m'\le m\)), then by Proposition \ref{3} we have \(m\le m'\); by antisymmetry of the lattice \(P\) it follows that \(m=m'\). Hence the claim.

    \end{proof}
    Combining Proposition \ref{1} and Lemma \ref{4}, we get the following corollary about the minimal element of the set $S_x^M.$
\begin{Corollary}\label{5}
   The minimal element of \(
S_x^M\) is $m.$
\end{Corollary}
We now obtain a minimal element of the union $\bigcup_{0\neq x\in M_i}S_x^M.$ First, we fix some definitions.

\begin{definition}
    Let \(x\in M_i\).  
Then we can write
\(
x = \sum_{\eta\in I_x} c_\eta t^{i-|u_\eta|} u_\eta.
\)
We call this expression the representation of \(x\) in terms of the basis elements \(u_\eta\).
\end{definition}

Define
\[
\mathcal{U}^a = \{\, u_\eta : u_\eta \text{ appears in the representation of } x_a^i \,\}, \] 
where $\{x^i_a\}_{a=1}^n$ is an $R$ basis of $M$ and \(\mathcal{U} = \bigcup_{a=1}^n \mathcal{U}^a.
\)
Let \( T \) be the set of join of degrees of all elements of the power set of \( \mathcal{U} \).  
Define
\[
T_1 = T \cap S, \text{ where } S = \bigcup_{x \in M_i,\, x \neq 0} S^M_x.
\]

Note that any nonzero \( x \in M_i \) can be written as a finite sum \( x = \sum \gamma_a x_a^i = \sum \gamma_a \sum c_\eta t^{i-|u_\eta|} u_\eta \) and  
the set \( T_1 \) is nonempty (since join of degrees of  \(\mathcal {U}^a \in T\cap S\) for all \(a\in \{1,2,\dots, n\}\)) and finite. Let \( \mu \) be one of its minimal elements.  
The following result yields a minimal element of $S$.
\begin{proposition}\label{6}
  The element \( \mu  \) is a minimal element of \(S\).
\end{proposition}
\begin{proof}
    As \( T_1 \subseteq S \), we have \( \mu \in S \).  
Suppose there exists \( \mu' \in S \) with \( \mu' \le \mu \).  \hfill (1)

Then there exists some \( 0 \neq x' \in M_i \) such that \( \mu' \in S^M_{x'} \).  
By Corollary \ref{5}, there exists 
\(
\mu'' = \min S^M_{x'} = \text{join of } \{|u_\eta|:\eta \in I_{x'}\}.
\)
Thus \( \mu'' \le \mu' \).  \hfill (2)

Moreover, \( \mu'' \in T \), since the set \( \{ u_\eta : \eta \in I_{x'} \} \) is in the power set of \(\mathcal{U} \).  
Also, it is the minimal element of \( S^M_{x'} \), so \( \mu'' \in S \).  
Hence \( \mu'' \in S \cap T = T_1 \).  
But we assumed that \( \mu \) is a minimal element of \( T_1 \), so \( \mu \le \mu'' \).  
From inequality 2 and the transitivity of \( P \), we get \( \mu \le \mu' \).  \hfill (3)

From inequalities 1 and 3, and by antisymmetry of \( P \), it follows that \( \mu = \mu' \).  
Hence \( \mu \) is a minimal element of \(S \).

\end{proof}
Using the previous proposition, we now show that the hypothesis of Theorem \ref{0} forces some $M_s$ to be zero, if $P$ has no minimal element.
\begin{proposition}\label{8}
If $s < \mu$, then $M_s = 0$.
\end{proposition}
\begin{proof}
If $s < \mu$, then $s < i$. Suppose there exists $k \neq 0$ in $M_s$.  
The image of $k$ under the map $M_s \to M_i$ is nonzero (by Lemma \ref{7}), implying $s \in S$, which contradicts the minimality of $\mu $ in $S$.
\end{proof}
For $r \in P$, define
\[
D_r = \sum_{q < r} \mathrm{Im}(M_q \to M_r),
\]
that is the sum of all images coming from $M_q$ to $M_r$ for $q < r$.  
Note that by the Proposition \ref{8}, $D_\mu = 0$.
Similarly, one can define \( D'_r \) for \( F_r \).

\begin{definition}\label{3.11}
Let $x\in M_i$ for some $i\in P$.
We say that $x$ is \emph{decomposable} if $x\in D_i$,
and \emph{indecomposable} if $x\in M_i\setminus D_i$.
\end{definition}
As a direct consequence of the last proposition, we prove the existence of an indecomposable element in $M$.

\begin{Corollary}\label{9}
There exists $w \le i$ such that $D_w \subsetneq M_w$.
\end{Corollary}

\begin{proof}
Take $w = \mu$. Either $\mu$ is a minimal element of $P$, or there is $s<\mu.$ 
In both the cases $D_\mu = 0 \subsetneq M_\mu$, as required.
\end{proof}
In the remaining of this section, we will fix a basis among the indecomposables and will show that the total number of such bases is finite when the degrees are less than a fixed degree $i$ and their images are linearly independent in $M_i$.

We know that for any \( r \in P \), both \( F_r \) and \( M_r \) are \( R \)-modules.  
As image spaces form submodules, the direct sums of image spaces \( D_r' \) and \( D_r \) are also submodules of \( F_r \) and \( M_r \) respectively.  
Hence, one can define the quotient \( R \)-modules \( F_r / D_r' \) and \( M_r / D_r \). And there exists at least some \( r \) for which these quotient modules are nontrivial, by Corollary \ref{9}.

\begin{proposition}\label{10}
Both $F_r / D_r'$ and $M_r / D_r$ are free $R$-modules, provided both $F_r$ and $M_r$ are finitely generated as $R$-modules.
\end{proposition}
\begin{proof}
    Take a nonzero element, say \( \bar{0} \neq \bar{\alpha} \in F_r / D_r' \).  
So \( \alpha \notin D_r' \).  
If possible, suppose there exists \( g \neq 0 \) in \( R \) such that \( g \bar{\alpha} = \bar{0} \); that is, \( g \alpha \in D_r' \).  
Then, as an element of \( F \), we can write 
\(
g\alpha = \sum_{\eta\in I_{g\alpha}} c_\eta t^{r-|u_\eta|}u_\eta, \quad c_\eta \in R,
\)
such that \( |u_\eta| < r \) for all \( \eta\).

Again, as an element of \( F \),
\(
\alpha = \sum_{\eta\in I_\alpha} d_\eta t^{r-|u_\eta|} u_\eta, \quad d_\eta \in R,
\)
and since \( \alpha \notin D_r' \), it follows that \( |u_\eta| = r \) for at least some \( \eta \).  
Now,
\(
g\alpha = \sum_{\eta\in I_{g\alpha} }g d_\eta t^{r-|u_\eta|} u_\eta.
\)
Since $g\neq 0$, $d_\eta\neq 0$, it follows that $g d_\eta \neq 0$ as $R$ is a PID, so the degrees \( |u_\eta| = r \) for those same indices as in the representation of $\alpha$.  
This contradicts the condition that all \( |u_\eta| < r \) for \( g\alpha \in D_r' \).  
Hence, such \( g \) does not exist. Therefore, \( F_r / D_r' \) is torsion-free.  

By the structure theorem for finitely generated modules over a principal ideal domain, \( F_r / D_r' \) is a free \( R \)-module.

We now consider \( M_r / D_r \).  
Let \( \bar{0} \neq \bar{\beta} \in M_r / D_r \), such that \( \beta \in M_r \) but \( \beta \notin D_r \).  
As an element of \( F \),
\(
\beta = \sum_\eta k_\eta t^{r-|u_\eta|}u_\eta, \quad k_\eta \in R.
\)
We claim that for at least some \( \eta \), \( |u_\eta| = r \).  
If possible, suppose \( |u_\eta| < r \) for all \( \eta \).  
We can write \( u_\eta = m_\eta + n_\eta \), where \( m_\eta \in M_{|u_\eta|} \) and \( n_\eta \in N_{|u_\eta|} \).  
Then,
\[
\beta = \sum_\eta k_\eta t^{r-|u_\eta|}(m_\eta + n_\eta) = \sum_\eta k_\eta t^{r-|u_\eta|} m_\eta + \sum_\eta t^{r-|u_\eta|}k_\eta n_\eta.
\]
So
\(
\beta - \sum_\eta k_\eta t^{r-|u_\eta|} m_\eta = \sum_\eta k_\eta t^{r-|u_\eta|} n_\eta
\)
belongs to \( M \cap N = \{0\} \).  
Hence,
\(
\beta = \sum_\eta k_\eta t^{r-|u_\eta|} m_\eta,
\)
where \( |m_\eta| = |u_\eta| < r \) for all \( \eta \), which implies \( \beta \in D_r \), a contradiction.  
Therefore, our claim holds.  

The rest of the argument follows analogously to the case of \( F_r / D_r' \), and thus \( M_r / D_r \) is also a free \( R \)-module.

\end{proof}
Since we have assumed $M_i$ is finitely generated $R$-module for all $i\in P$, we can apply Proposition \ref{10} to give the following definition.
\begin{definition}
   \[ b_r = \text{dim } (M_r / D_r) = \text{dim } M_r - \text{dim } D_r.
\]
\end{definition}

Let us choose \( e_j^r \) of degree \( r \) for \( j = 1, 2, \ldots, b_r \) from \( M_r - D_r \) such that the corresponding \( \{\overline{e_j^r}:j = 1, 2, \ldots, b_r \} \) forms a basis of \( M_r / D_r \).
\begin{proposition}\label{11}
  The set $\{\overline{e_j^r}: j = 1,\ldots,b_r\}$ is linearly independent in $F_r / D_r'$.
\end{proposition}
\begin{proof}
    For the linearly independent set
\(
\{\overline{e_j^r}\mid j=1,\dots,b_r\}\subset M_r/D_r,
\)
consider a relation in \(F_r/D_r'\):
\(
\sum_j d_j\,\overline{e_j^r}=\overline{0}\ ,\text{ where   }d_j \in R.
\)

This means \(\sum_j d_j e_j^r\in D_r'\). By definition of \(D_r'\), there exist elements \(x_i\in F_{s_i}\) with \(s_i=|x_i|<r\) such that
\[
\sum_{i=1}^{h} t^{\,r-s_i} x_i = \sum_j d_j e_j^r. \tag{1}
\]
Write \(x_i=m_i+n_i\) with \(m_i\in M_{s_i}\) and \(n_i\in N_{s_i}\). Substituting into (1) gives
\[
\sum_{i=1}^h t^{\,r-s_i} n_i \;=\; \sum_j d_j e_j^r \;-\; \sum_{i=1}^h t^{\,r-s_i} m_i.
\]
We have 
\(
\sum_j d_j e_j^r = \sum_{i=1}^h t^{\,r-s_i} m_i \in D_r,
\)
so
\(
\sum_j d_j\,\overline{e_j^r}=\overline{0}\quad\text{in }M_r/D_r.
\)
As \(\{\overline{e_j^r}\}\) is linearly independent in \(M_r/D_r\), we conclude \(d_j=0\) for all \(j\). Therefore \(\{\overline{e_j^r}\mid j=1,\dots,b_r\}\) is also linearly independent in \(F_r/D_r'\).

\end{proof}

Note that since $\{\overline{e_j^r}\}$ is linearly independent in $M_r/D_r$, the set consisting of the union 
\(
\bigcup_{r < i} \bigcup_{j=1}^{b_r} \overline{e_j^r}
\)
is linearly independent in the direct sum $\bigoplus_{r < i} M_r/D_r$. We denote the set
\[
\bigcup_{r<i} \bigcup_{j=1}^{b_r} e_j^r
\]
by \(\mathcal{B}^i\). Define \[
\mathcal{B}=\bigcup_{k \in P} \bigcup_{j=1}^{b_r} \{ e_j^r \}.
\]

Proposition \ref{11} implies the following directly.
\begin{Corollary}\label{12}
  The set   \(
\bigcup_{r < i} \bigcup_{j=1}^{b_r} \overline{e_j^r}
\) is linearly independent in the direct sum 
\(
\bigoplus_{r < i} F_r / D_r'.
\)

\end{Corollary}

 We now consider the images of the sets of $\mathcal B^i$ in $M_i.$
\newpage
\begin{proposition}\label{13}
    The image set of \(\mathcal{B}^i\) in \(M_i\), that is
\[
\Big\{\, t^{\,i-r} e_j^r \ \Big|\ r < i,\ 1 \leq j \leq b_r \,\Big\}
\]
is linearly independent in \( M_i \).

\end{proposition}
\begin{proof}
    
We fix \(r\) and \(j\). We now express 
\begin{equation} \label{eq:ejr}
e_j^r = \sum_{l} c_{jl}^r\, u_l^r + d_j^r,
\end{equation}
where \(u_l^r\) are homogeneous basis elements of \(F\) of degree \(r\),
\(c_{jl}^r \in R\), and \(d_j^r \in D_r'\).
Suppose, if possible, $\{t^{i-r}e_j^r:r<i,\; 1\le j\le b_r\}$ is not linearly independent. then there are finitely many $r,\, j$ and $\lambda_j^r\in R$ such that 
\(
\sum_{r < i} \sum_j \lambda_j^r\, t^{\,i-r} e_j^r = 0 \quad \text{in } F_i.
\)
Substituting from \eqref{eq:ejr} gives
\[
\sum_{r < i} \sum_j \lambda_j^r\, t^{\,i-r} 
\left( \sum_l c_{jl}^r u_l^r + d_j^r \right) = 0.
\]
In the above expression, we also write $d_j^r$ in the representation in terms of the basis elements $\{u_\eta\}$ as $d_j^r=\sum \overline{c}_{jl}^ru_{jr}^k$, where $|u_{jr}^k|<|d_j^r|,\, \overline{c}_{jl}^r\in R[U_0].$
Define 
\[
S = \bigl\{\, \deg(u_l^r),\;\deg(u_{jr}^k) \ \big|\ r,l,j \text{ appear in the above sum} \bigr\}.
\]
This set is finite. Let \(\sigma\) be a maximal element of $S$. 
Note that the $u_\eta$'s present in the representation of each of $d^r_j$ of the corresponding $e^r_j$, where $r<\sigma$ cannot contribute to degree \(\sigma\).

It follows that, for all
\(l,\,
t^{i-\sigma}(\sum_{j} \lambda_j^\sigma\, c_{jl}^\sigma) = 0,\text{ in } R[U_0] \text{, thus } \sum_{j} \lambda_j^\sigma\, c_{jl}^\sigma = 0.
\)
Hence
\(
\sum_{l} \sum_{j} 
\lambda_j^\sigma\,  c_{jl}^\sigma\, u_l^\sigma = 0.
\)
Equivalently,
\(
\sum_{j} \lambda_j^\sigma 
\left( \sum_{l} c_{jl}^\sigma u_l^\sigma \right) = 0
\quad \text{, therefore} \quad
\sum_{j} \lambda_j^\sigma \left( e_j^\sigma - d_j^\sigma \right) = 0\text{   (from \ref{eq:ejr}).}
\)
Thus
\(
\sum_{j} \lambda_j^\sigma\, \overline{e_j^\sigma} = \overline{0}
\quad \text{in} \quad \bigoplus_{r < i} F_r / D_r'.
\)
By the Corollary \ref{12}, this implies \(\lambda_j^\sigma = 0\) for all such \(j\).

We now remove \(\sigma\) from \(S\), and repeat the same argument for a maximal element of \(S-\{\sigma\}\) using the following expression \[
\sum_{r\neq\sigma,r < i} \sum_j \lambda_j^r\, t^{\,i-r} 
\left( \sum_l c_{jl}^r u_l^r + d_j^r \right) = 0.
\] Since \(S\) is finite, after finitely many steps 
we obtain \(\lambda_j^r = 0\) for all \(r < i\) and all \(j\). This shows that the image of \(\mathcal{B}^i\) is linearly independent in \(F_i\).
Since \(M_i\) is a submodule of \(F_i\), the proof is complete.

\end{proof}
We now show that the cardinality of the elements in $\mathcal{B}^i$ is finite. This result will be applied to writing any element $x\in M_i$ `in terms' of $\{e_j^r\}$.
\newpage
\begin{theorem}\label{14}
  The sum \(
\sum_{r<i} b_r 
\) is finite.
\end{theorem}
\begin{proof}
    Suppose, to the contrary, that 
\(
\sum_{r<i} b_r = \infty.
\)
Then there exists an infinite linearly independent family 
\(\{\overline{e_j^r}\}_{r<i,\; j\ge1}\) in 
\(\bigoplus_{r<i} M_r/D_r\). By Corollary~\ref{12} this family is also linearly independent in 
\(\bigoplus_{r<i} F_r/D_r'\), and by Proposition~\ref{13} its image 
\(\{t^{\,i-r}e_j^r\}_{r<i,\; j\ge1}\) is linearly independent in \(M_i\). Thus \(\{t^{\,i-r}e_j^r\otimes 1\}_{r<i,\; j\ge1}\) is linearly independent in $M_i\otimes_RF$, where $F$ is the fraction field of $R$. This contradicts the hypothesis that $\dim_F(M_i\otimes F)=\dim_R M_i<\infty.$ Hence 
\(\sum_{r<i} b_r< \infty\).

\end{proof}
For the rest of the article, when we refer to an $R[U_0]$-homomorphism between two graded $R[U_0]$-modules (that is persistence modules), we always consider a graded $R[U_0]$-module homomorphism.  
\section{Basis for the projective module}\label{sec4}
In this section,  we show that $\mathcal{B}=\cup_{r\in P}\cup_{j=1}^{b_r}e_j^r$ gives an $R[U_0]$-basis for $M.$ The first step is to show that every element in a projective module `can be written in terms of indecomposables'. In the next section this result will be applied to prove that a certain family of modules is not projective.
For any $i\in P$, define
\[
M_{\langle i\rangle}=R[U_0]\cdot M_i 
\] to be a free \(R[U_0]\)-module with basis (in degree \(i\)) \(\widetilde{m}_j\), where \(\widetilde{m}_j\) corresponds to a basis element
\(m_j\in M_i\).

Define an $R[U_0]$-module homomorphism
\(
\phi : M_{\langle i\rangle} \longrightarrow M
\)
by
\[
\phi\!\left(t^{\,r-i}\sum_j c_j \widetilde{m}_j\right)
=
t^{\,r-i}\sum_j c_j m_j ,
\]
where \(c_j\in R\).

This induces an $R[U_0]$-module homomorphism
\(
q:\bigoplus_{i \in P} M_{\langle i\rangle}\longrightarrow M
\)
given by coordinatewise sum.  
Then $q$ is onto, since every $x\in M$ is the image of the element
\[
0+\cdots+0+x+0+\cdots\in \bigoplus_{i\in P} M_{\langle i\rangle}
\]
where $x$ appears in the $M_{\langle i\rangle}$–coordinate.

As $M$ is projective, there exists a lift
\(
p:M\longrightarrow \bigoplus_{i \in P} M_{\langle i\rangle}
\)
such that
\[
\begin{tikzcd}
\displaystyle\bigoplus_{i\in P} M_{\langle i\rangle} \arrow[r,"q"] 
& M \arrow[r] 
& 0 \\
& M \arrow[u,"\mathrm{id}_M"{xshift=0.7ex}] 
      \arrow[ul,"p"'{yshift=-1.2ex}]
\end{tikzcd}
\]

 by Definition \ref{def}.
 We make the following two observations on an $R[U_0]$-module homomorphism from $M$ to $M_{\langle k\rangle}.$
\begin{remark}\label{re1}
If for any \(R[U_0]\)-module homomorphism,
\(
f:M\longrightarrow M_{\langle k\rangle}
\)
we have
\(
f(x)=t^{\,k-|y_\ell|}y_\ell \neq 0,\, |y_\ell|=k,
\)
then
\(
|x|=|f(x)|\ge k,
\)
since \(f\) is graded.
\end{remark}

\begin{remark}\label{re2}
If \(x\) is decomposable, say \(x\in D_k\), then any \(R[U_0]\)-module homomorphism
\(
f:M\longrightarrow M_{\langle k\rangle}
\)
satisfies \(f(x)=0\).
\end{remark}
\begin{proof}
Let \(
x=\sum_r t^{\,k-|d_r|} d_r,
\; |d_r|<k \text{ for all } r.
\)
Then \(f(d_r)=0\) for all \(r\), since \(f\) preserves degree and
\(
(M_{\langle k\rangle})_{|d_r|}=\{0\} \quad \text{for all } r.
\)
Hence \(f(x)=0\).
\end{proof}

Combining Remark \ref{re1} and \ref{re2}, we get a necessary and sufficient criterion for an element to be decomposable.

\begin{lemma}\label{lem}

An element \(x\in M\) is decomposable if and only if 
\(
p(x)=\ell_1+\ell_2+\cdots+\ell_r
\qquad\text{with } 0\neq \ell_j\in M_{\langle s_j\rangle},\ s_j<|x|\ \text{for all }j .
\)
\end{lemma}
\begin{proof}
Let
\[
p(x)=\sum_{i=1}^r \ell_i,
\qquad 
\ell_i=t^{\,k-s_i}z_i \ \text{with } |z_i|=s_i<k .
\]
Since \(x=q(p(x))\), where \(q\) is the sum map, it follows that
\(x\in D_k\).

Conversely, let \(x\in D_k\).
Then define
\(
f_i=\pi_i\circ p : M \longrightarrow M_{\langle i\rangle},\text{ for those }i \text{ with }\ell_i\neq 0.
\)
If \(f_i(x)\neq 0\), then by the above remarks we must have \(i<k\).

\end{proof}
\begin{definition}
    We say that a decomposable element $x\in M_k$ \emph{can be written in terms of indecomposable elements}, if
\(
x=\sum_{i=1}^m t^{\,k-\ell_i}x_i,
\)
where each $x_i$ is indecomposable element in $ M_{\ell_i}$ and $\ell_i<k$.

\end{definition}
We now obtain a necessary condition on the decomposable elements of a projective module.
\begin{lemma}\label{propp}
Any decomposable element can be written in terms of indecomposable elements.
\end{lemma}
\begin{proof}
    Let $x\in D_k.$ Let $P:M\to \oplus_{i\in P}M_{<i>}$ be the map defined at the beginning of this section.
Let
\(
p(x)=0+\cdots+0+\ell_1+\ell_2+\cdots+\ell_r+0+\cdots
\in \bigoplus_{i\in P} M_{\langle i\rangle},
\)
where $0\neq \ell_j\in M_{\langle s_j\rangle}$ and
\(
\ell_j=t^{\,k-s_j}z_j \quad \text{with } |z_j|=s_j<k
\)
by Lemma~\ref{lem}.
Thus
\(
x=q(p(x))=\ell_1+\ell_2+\cdots+\ell_r .
\)

Let $s_i$ be a maximal element among $\{s_1,\dots,s_r\}$.
We claim that $z_i$ is indecomposable. Suppose, to the contrary, that $z_i$ is decomposable.
Then
\[
p(z_i)=\sum_\alpha h_\alpha, \tag{A}
\]
where $0\neq h_\alpha\in M_{\langle \bar{\alpha}\rangle}$ and
$\bar{\alpha}<s_i$ for all $\alpha$. Now
\(
x=\sum_j \ell_j=\sum_j t^{\,k-s_j}z_j,\text{ implies }
p(x)=\sum_j t^{\,k-s_j}p(z_j),
\)
since $p$ is an $R[U_0]$-module homomorphism. For $j\neq i$, write
\(
p(z_j)=\sum_\beta a_\beta,
\)
where $0\neq a_\beta\in M_{\langle \bar{\beta}\rangle}$ with
$\bar{\beta}\le s_j$.

\makebox[0pt][r]{(B)}
\bigg\{
\begin{minipage}{.8\textwidth}
    If for some $j\neq i$ the component of $M_{\langle s_i\rangle}$ in $p(z_j)$
was nonzero, then we would have $s_i\le s_j$, contradicting the maximality
of $s_i$.
Hence, in each $p(z_j)$ the component in $M_{\langle s_i\rangle}$ is zero.
\end{minipage}

Combining (A) and (B), we observe that the component of $M_{\langle s_i\rangle}$ in $p(x)$ is zero, which contradicts our initial assumption that $z_i$ is decomposable.

We now can write
\(
\ell_i+\sum_{s_j<s_i}\ell_j
=
t^{\,k-s_i}z_i+\sum_{s_j<s_i}t^{\,k-s_j}z_j
=t^{\,k-s_i}z_i+\sum_{s_j<s_i}t^{\,k-s_i}t^{s_i-s_j}z_j=
t^{\,k-s_i}\left(
z_i+\sum_{s_j<s_i}t^{\,s_i-s_j}z_j
\right).
\)
Here
\(
z_i+\sum_{s_j<s_i}t^{\,s_i-s_j}z_j
\)
is also indecomposable, as $z_i$ is indecomposable and $\sum_{s_j<s_i}t^{\,s_i-s_j}z_j$ is decomposable.

Repeating this argument for each maximal $s_i$, we conclude that
any $x\in D_k$ can be written in terms of indecomposable elements.

\end{proof}
\begin{remark}\label{remnew}
    Note that Remark \ref{re1}, \ref{re2}, Lemma \ref{lem}, \ref{propp} are valid when $P$ is a preordered set with compatible abelian group structure. \end{remark}
The next result provides a basis for the $R[U_0]$-module $M.$
\begin{proposition}\label{15}
   The union  \(
\mathcal{B}=\bigcup_{k \in P} \bigcup_{j=1}^{b_r} \{ e_j^r \}
\)
forms a basis of \( M. \)

\end{proposition}
\begin{proof}
\textbf{Any element $x$ can be written as an $R[U_0]$-combination of elements of $\mathcal{B}$:}\\
If $x \in M_k$ but $x \notin D_k$, then \( \overline{0} \neq \overline{x}  = \sum_j c_j \, \overline{e_j^k} . \)
This implies 

\(
x - \sum_j c_j e_j^k \in D_k,
\)
say \( y = x - \sum_j c_j e_j^k \in D_k. \)
then
\(
x = y + \sum_j c_je_j^k,
\)
where $c_j \in R$.  
By the Lemma \ref{propp}, the element \(y\) can be written in terms of indecomposable elements, that is \(
y=\sum_i t^{\,k-|y_i|}\,y_i,\quad|y_i|<k
\) where each $y_i\in M_{|y_i|}$ is indecomposable.

Each \(y_i\) is in the same situation as \(x\) but with strictly smaller degree.
Repeating this procedure, the process terminates after finitely many steps
by Theorem~\ref{14}.
In the obvious case when \(x\in D_k\), we may view \(x\) itself as \(y\), and the
argument is the same.

Therefore, in both the cases, $x$ can be represented as a linear combination of the basis elements $e_j^r$ with coefficients from $R[U_0]$.

\textbf{Uniqueness of the expression:}
Let $x \in M_k $.  Let \(x=y+z=y'+z'\), where \(y,y'\in D_k\) and \(z,z'\in M_k\setminus D_k\).
Write
\(
z=\sum_{\ell} d_\ell\, e_\ell^{\,k}
\;\text{and}\;
z'=\sum_{j} c_j\, e_j^{\,k}.
\)

Then $y - y' = z' - z$, Since \(y-y'\in D_k\), we have
\[
z' - z = \sum_j c_j e_j^k - \sum_\ell d_\ell e_\ell^k
       = \sum_i a_i e_i^k \text{ [say] } \in D_k. 
\]
Hence
\(
\overline{\sum_i a_i e_i^k} \;=\; \sum_i a_i \overline{e_i^k} \;=\; \overline{0}.
\)
implying all $a_i = 0$ since $\{\overline{e_i^k}\}$ is linearly independent in $M_k / D_k$. Hence \(z-z'=0.\)
Thus, $z = z'$ and $y = y'$.

Suppose \(y\) admits two representations
\[
\sum_{\alpha} c_\alpha t^{k - |e_\alpha|} e_\alpha
=
\sum_{\beta} d_\beta t^{k - |e_\beta|} e_\beta.
\]
So
\(
\sum_{\delta\in \alpha\cup\beta} (c_\delta-d_\delta)\,
t^{\,k-|e_\delta|} e_\delta = 0 ,\;\text{where } |e_\delta|<k.
\)
By Proposition~\ref{13}, we obtain \(c_\delta=d_\delta\) for all \(\delta\),
and hence the representation is unique.

\end{proof}
\subsection{Proof of the Theorem \ref{0}:}\label{sec7}
If \( M \) is graded free, then it is a direct summand of a graded free module (namely \( M \oplus 0 \)), hence graded projective. 

Conversely, we obtain a homogeneous basis for \(M\) by Proposition~\ref{15}, it follows that \(M\) is a graded free \(R[U_0]\)-module.

\section{A Family of flat persistence module, which are not projective}\label{sec5}

In this section,  we provide a family of flat modules. We obtain a criterion (Lemma-\ref{lemfalt}) for a persistence module to be not projective. The section ends with an example (\ref{exx}) of persistence module that are flat but not projective. These results are generalizations of Proposition 6.1 and Corollary 6.4 in \cite{bubenik2021homological}. For this section we assume  the indexing set $(P,\le)$ is a preordered set with compatible abelian group structure (due to Remark \ref{remnew}) unless mentioned otherwise.

\begin{lemma}\label{lemfalt}
Let $P$ be a lattice and $O\subset P$.
Assume there exists $a\in O$ such that no element of $V_a\cap O$ is a minimal
element of $O$, where $V_a=\{x \in P :x\le a\}$.

Let $R$ be a PID and $M$ be an $R[U_0]$-module, such that
\[
\bigg\{
\begin{aligned}
M_i &\neq 0 \text{, is a finite dimensional free } R \text{ module} , && i \in O, \\
M_i &= 0, && i \notin O .
\end{aligned}
\]

Then $M$ is not projective.
\end{lemma}

\begin{proof}
    Suppose, for contradiction, that $M$ is projective.
Since $a$ is not a minimal element of $O$, there exists $b<a$ with $b\in O$.
By Lemma~\ref{7}, the structure map $M_b\to M_a$ is injective.
Hence, there exists a decomposable element $z\in M_a$.

By Lemma \ref{propp}, we can write
\(
z=\sum_i t^{n_i} z_i,
\)
where $\deg(z_i)<\deg(z)=a$ and each $z_i$ is indecomposable.

Now $\deg(z_i)\in V_a\cap O$, and by assumption every element of $V_a\cap O$
is not a minimal element of $O$.
Therefore, for each $z_i$ there exists a decomposable element
$y\in M_{\deg(z_i)}$.
Thus we can write
\(
y=\sum_j t^{m_j} y_j,
\)
where
\(
\deg(y_j)<\deg(z_i)<a,
\)
and each $y_j$ is indecomposable.

Continuing this process, we obtain indecomposable elements in
infinitely many mutually distinct degrees.
Hence we obtain infinitely many linearly independent indecomposable elements.

This contradicts Theorem \ref{14}, that is
\(
\sum_{r<i} b_r < \infty.
\)
Therefore, $M$ cannot be projective.

\end{proof}

\begin{definition}
    Let $R$ be a PID and $A \subset P$ be a convex subset. The indicator persistence module on $A$ is the persistence module $R_A : P \to \mathrm{Mod}_R$ given by ${(R_A)}_a = R$ if $a \in A$ and $0$ otherwise, and all the maps ${(R_A)}_{a \le b}$, where $a,b \in A$, are identity maps.

\end{definition}
\begin{definition}
    Let $U_i$ be a principal upset and $R_{U_i}$ be a free $R[U_0]$-module with generator $x_i$, where $\deg(x_i)=i$.
Consider the family $\{R_{U_i}\}_{i\in U}$ with $R[U_0]$-module homomorphisms
\(
f_{ji} : R_{U_j} \to R_{U_i}
\)
induced by the inclusion $U_j \subseteq U_i$ for $i \le j$, defined by
\[
f_{ji}(x_j)=t^{\,j-i}x_i .
\]
Define
\[
\operatorname{colim}_{i\in U} R_{U_i}
:=
\left(\bigoplus_{i\in U} R_{U_i}\right)\big/ N,
\]
where $N$ is the submodule generated by
\(
\{x - f_{ji}(x)
: x\in R_{U_j},\; i\le j\}
\) (see Lemma 10.8.2, \cite{stacks-project}).
\end{definition}
The following result expresses the indicator module $R_U$ on an upset $U$ as a colimit of some indicator modules.
\begin{proposition}\label{propcolim}
Let $U\subset P$ be an upset, then
\(
\operatorname{colim}_{i \in U} R_{U_i} = R_U
\)
\end{proposition}
\begin{proof}
    Denote
\(
C=\operatorname{colim}_{i\in U} R_{U_i}.
\)
We write
\[
C=\bigoplus_{s\in P} C_s,
\text{ so }
C_s=\left(\bigoplus_{i\in U} (R_{U_i})_s\right)\big/ N_s,
\]
where
\(
N_s = N \cap \bigoplus_{i\in U} (R_{U_i})_s .
\)
\textbf{Case 1.} If $s\notin U$, then $(R_{U_i})_s=0$ for all $i\in U$, hence $C_s=0$.

\textbf{Case 2.} If $s\in U$, then
\[
(R_{U_i})_s =
\begin{cases}
R, & i\le s,\\
0, & \text{otherwise}.
\end{cases}
\]
Thus
\(
\bigoplus_{i\in U} (R_{U_i})_s
=
\bigoplus_{\substack{i\in U\\ i\le s}} R,
\)
and the relations defining $N_s$ identify all summands, so $C_s \cong R$.

We now denote the identity $1\in (R_U)_s=R$ by $y_s$.
We define an $R[U_0]$-module homomorphism from $R_{U_i}$ to $R_U$ by sending $x_i$ to $y_i$.
This induces an $R[U_0]$-module isomorphism from $C$ to $R_U$.

\end{proof}
We recall a sufficient condition for obtaining a flat module (see Lemma 6.3, \cite{bubenik2021homological}).
\begin{lemma}\label{lembub}
    Colimits of graded projective modules are $\otimes_{\mathrm{gr}}$-flat.

\end{lemma}
Combining Proposition \ref{propcolim} and Lemma \ref{lembub}, we establish the flatness of indicator modules $R_U$ for upsets $U$.
\newpage
\begin{Corollary}\label{cor}
For any up-set $U$ of $P$, $R_U$ is a graded flat persistence module over $R[U_0]$.

\end{Corollary}

\begin{proof}
Since every indicator module for any principal upset is graded free, the proof is obvious.

\end{proof}
Since a direct sum of flat modules is flat, we generalize Corollary \ref{cor} to attain a more general result.
\begin{Corollary}\label{5.7}
    If $\{U_l\}_{l\in L}$'s are upsets in $P$, then $\oplus_{l \in L} R_{U_l}$ is a flat module.
\end{Corollary}

Below we give a picture (Figure:\ref{fig:simple-tikz}) of flat modules which are not projective. 
\begin{example}\label{exx}

    Using the notation of Corollary~\ref{5.7}, we consider a family of $U_\ell$ and $\overline{U}_\ell$'s each is an upset defined in the following.
Let $\mathcal{L}$ be the union of the negative $x$--axis and the negative $y$--axis in $\mathbb{R}^2$.
For each $\ell \ge 0$, let $\mathcal{L}_\ell$ denote the translation of $\mathcal{L}$ by $(\ell,\ell)$.

Define $\overline{U}_\ell$ to be the upset
\[
\overline{U}_\ell
=
\bigl\{(x,y) \in \mathbb{R}^2 \mid (x,y) \ge (x_0,y_0) \text{ for some } (x_0,y_0) \in \mathcal{L}_\ell \bigr\},
\]
which is an upset.
Set
\(
U_\ell := \overline{U}_\ell \setminus \mathcal{L}_\ell .
\) Then each of $R_{U_\ell}$ and $R_{\overline{U}_\ell}$ is flat (by Corollary~\ref{cor}), and $R_{U_\ell}$ is not projective (using Lemma~\ref{lemfalt}), by taking $a$ to be $(\ell+1,\ell+1)$.
We now let $L_1, L_2 \subseteq \mathbb{R}_{\ge 0}$ be such that
$L_1 \cap L_2 = \varnothing$ and $0 \in L_1$.
In the following picture, we consider the direct sum
\[
\left( \bigoplus_{\ell \in L_1} R_{U_\ell} \right)
\;\oplus\;
\left( \bigoplus_{\ell \in L_2} R_{\overline{U}_\ell} \right).
\]
By Corollary~\ref{5.7}, this yields a flat module, which is not projective due to the presence of $R_{U_0}$.

   \begin{figure}[H]
\centering
\resizebox{0.42\textwidth}{!}{%
\begin{tikzpicture}
\begin{axis}[
    axis lines=middle,
    xmin=-4, xmax=6,
    ymin=-4, ymax=6,
    xlabel={$x$},
    ylabel={$y$},
    ticks=none,
    width=10cm,
    height=10cm,
]

\addplot[pattern=grid, pattern color=black, draw=none]
coordinates {(-4,3) (6,3) (6,6) (-4,6)};
\addplot[pattern=grid, pattern color=black, draw=none]
coordinates {(3,-4) (6,-4) (6,3) (3,3)};

\addplot[pattern=crosshatch, pattern color=black, draw=none]
coordinates {(-4,2) (6,2) (6,6) (-4,6)};
\addplot[pattern=crosshatch, pattern color=black, draw=none]
coordinates {(2,-4) (6,-4) (6,2) (2,2)};

\addplot[pattern=north east lines, pattern color=black!50, draw=none]
coordinates {(-4,1) (6,1) (6,6) (-4,6)};
\addplot[pattern=north east lines, pattern color=black!50, draw=none]
coordinates {(1,-4) (6,-4) (6,1) (1,1)};

\addplot[pattern=dots, pattern color=black, draw=none]
coordinates {(-4,0.05) (6,0.05) (6,6) (-4,6)};
\addplot[pattern=dots, pattern color=black, draw=none]
coordinates {(0.05,-4) (6,-4) (6,0) (0.05,0)};

\addplot[dashed, thick] coordinates {(-4,0) (6,0)};
\addplot[dashed, thick] coordinates {(0,-4) (0,6)};

\end{axis}
\end{tikzpicture}
}
\caption{Flat but not projective modules}
\label{fig:simple-tikz}
\end{figure}
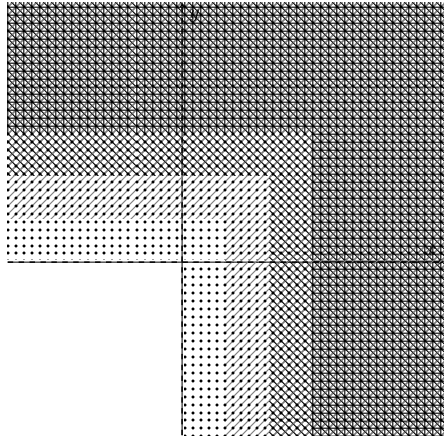
\end{example}

\section{Algorithms}\label{sec6}
In this section, we provide two algorithms to check whether a given persistence module over the field of real numbers $\mathbb{R}$, with indexing set $\mathbb{Z}$ or $\mathbb{Z}^2$, is graded free and to compute a basis. The construction of a basis is similar to that in Theorem~\ref{0}. In this section, however, the hypothesis of projectiveness is replaced by two sets of sufficient criteria. Under these criteria, we first prove graded freeness and then present the algorithms together with their complexities.

\subsection{\texorpdfstring{When the indexing set is $\mathbb{Z}$}
{When the indexing set is Z}}

If $M$ is a persistence module over $\mathbb{R}$ with $\mathbb{Z}$ as the indexing set, then $M$ is an $\mathbb{R}[t]$--module in the terminology of $\S$~\ref{sub23}, with $U_0=\mathbb{Z}_{\ge 0}$.

\begin{proposition}\label{6.1prop}
Let 
\(
M=\bigoplus_{i\in\mathbb{Z}} M_i
\) be a persistence module over $\mathbb{R}$ with the indexing set $\mathbb{Z}$ such that $\dim_{\mathbb{R}} M_i=d_i<\infty$ for all $i$. Assume:
\begin{enumerate}
\item $M_i=0$ for $i<\alpha$;
\item for any $i$, $f_{i,i+1}:M_i\to M_{i+1}$ is injective;
\item there exists $\beta \in\mathbb{Z}$ such that for $i>\beta$, $f_{i,i+1}$ is an isomorphism.
\end{enumerate}
Then $M$ is graded free.
\end{proposition}

\begin{proof}
The persistence module $M$ is a module over $\mathbb{R}[t]$. For any $r\in\mathbb{Z}$ define
\(
D_r=\sum_{k<r}\operatorname{Im}(M_k\to M_r)=\operatorname{Im}(M_{r-1}\to M_r).
\)
There exists a subspace $C_r$ such that $M_r=D_r\oplus C_r$.  
Let $\{e^r_1,e^r_2,\dots,e^r_{b_r}\}$ be indecomposable elements as in Definition~\ref{3.11}, forming a basis of $C_r$.  
Due to (1) and (3), the set 
\(
\mathcal{B}=\{e^r_j:1\le j\le b_r,\ r\in\mathbb{Z}\}
\)
is finite.
Let $x\in M_k\setminus D_k$. Then there exist coefficients $c^k_j$ such that
\(
x-\sum_{j=1}^{b_k} c^k_j e^k_j=y\in D_k.
\)
If $y\neq0$, then there exists $x'\in M_\ell$ with $\ell<k$ such that
\(
t^{k-\ell}x'=y,
\; x'\notin D_\ell.
\)
We repeat the same process for $x'$. This stops after finitely many steps since $M_i=0$ for all $i<\alpha$.  
If $x\in D_k$, we apply the same argument. Hence any element of $M_k$ can be written in terms of the elements of $\mathcal{B}$. Therefore $M$ is finitely generated over $\mathbb{R}[t]$.
Since each $f_{i,i+1}$ is injective, it follows that $f_{i,j}$ is injective for $i<j$. This shows that $M$ is torsion free. Hence, by the fundamental theorem of finitely generated modules over a PID, $M$ is free.
It remains to show that $\mathcal{B}$ is a basis.

It is enough to show that
\[
\sum_{j=1}^{m} c^r_j\, t^{\,i-r}e^r_j=0
\tag{$*$}
\]
implies all $c^r_j=0$.
We prove this by induction on $m$. Assume the statement holds for all natural numbers less than $m$.  
Order the elements $e^r_j$ in ascending degree and let the last elements
\(
e^{\sigma}_{m-s+1},\dots,e^{\sigma}_{m}
\)
have the highest degree $\sigma$. From $(*)$ we obtain
\(
t^{\,i-\sigma}\Big(\sum_{j=1}^{m} c^r_j\, t^{\sigma-r}e^r_j\Big)=0,
\)
hence
\(
\sum_{j=1}^{m} c^r_j\, t^{\sigma-r}e^r_j=0
\)
by injectivity. Thus
\[
\sum_{j=m-s+1}^{m} c^\sigma_j e^\sigma_j
=
-\sum_{j=1}^{m-s} c^r_j\, t^{\sigma-r}e^r_j
\in D_\sigma.
\]
Passing to $M_\sigma/D_\sigma$ gives
\(
\sum_{j=m-s+1}^{m} c^\sigma_j \,\overline{e^\sigma_j}=0,
\)
so $c^\sigma_j=0$ for $j=m-s+1,\dots,m$. Hence
\(
\sum_{j=1}^{m-s} c^r_j\, t^{\sigma-r}e^r_j=0,
\)
and by the induction hypothesis all $c^r_j=0$.
Therefore $\mathcal{B}$ is a basis of $M$, and $M$ is graded free.
\end{proof}

We now state the algorithm to obtain a basis according to the construction given in Proposition~\ref{6.1prop}.

\begin{algorithm}[H]\label{1d}
\caption{Find an $\mathbb{R}[t]$-basis of $M$}
\KwIn{Matrices ${(A_i)}_{d_{i+1}\times d_i}$ for $\alpha\le i\le \beta$. }

\KwOut{An $\mathbb{R}[t]$-basis $B$ of $M$}

Initialize $B\leftarrow$ columns of $I_{d_\alpha}$\;

\For{$i=\alpha,\alpha+1,\dots,\beta$}{
    $R_i|E_i \leftarrow \operatorname{RREF}(A_i|I_{d_{i+1}})$\;
    $r_i \leftarrow$ number of nonzero rows of $R_i$\;

    \eIf{$r_i<d_i$}{
        \textbf{stop algorithm}\;
    }{
        $P_i \leftarrow$ truncate first $r_i$ columns from the left of $I_{d_{i+1}}$\;
        $B \leftarrow B \cup \big(E_i^{-1}\times$ columns of $P_i\big)$\;
    }
}
\Return{$B$}
\end{algorithm}
\begin{algorithm}[H]
\caption{RREF$(A)$}
\KwIn{$A\in\mathbb{R}^{m\times n}$}
\KwOut{$A$ in RREF}

\For{$i=1$ \KwTo $m$}{
 \If{row $i\neq 0$}{
  scale row $i$ to create pivot\;
  clear pivot column in all other rows\;
 }
}
swap rows to obtain echelon ordering\;
\Return{$A$}
\end{algorithm}
We now compute the complexity of the above algorithms.
\begin{proposition}
Let
\(
N=\beta-\alpha+1
\)
denote the number of graded indices and let
\(
D=\max_{\alpha\le i\le\beta} d_i .
\)
Then Algorithm~\ref{1d} terminates after at most
\(
O(ND^{3})
\)
arithmetic operations over $\mathbb{R}$.
\end{proposition}

\begin{proof}
We analyze the algorithm step by step. The initialization
\(
B \leftarrow \text{columns of } I_{d_\alpha}
\)
requires constructing an identity matrix of size at most $D$,
which costs $O(D^{2})$ operations.
The loop runs over $N=\beta-\alpha+1$ indices.
For each $i$, the algorithm computes
\(
\operatorname{RREF}(A_i \mid I_{d_{i+1}}).
\)
Here
\(
A_i \in \mathbb{R}^{\,d_{i+1}\times d_i},
\)
and therefore the augmented matrix has size
\(
d_{i+1}\times (d_i+d_{i+1}).
\)
Since $d_i,d_{i+1}\le D$, this matrix has at most
\(
D \text{ rows and } 2D \text{ columns}.
\)
Computing the row-reduced echelon form of such a matrix via
Gaussian elimination requires $O(D^{3})$ arithmetic operations (see \cite{jacquard2023barcode}).
The additional steps in each iteration consist of counting nonzero rows, truncating columns of an identity matrix, multiplying $E_i^{-1}$ with selected columns and updating the set $B$; each bounded by $O(D^{2})$ operations.
These are dominated by the RREF computation.

Since the loop executes $N$ times and each iteration costs
$O(D^{3})$ operations, the overall complexity is
\(
O(ND^{3}).
\)
\end{proof}

\subsection{\texorpdfstring{When the indexing set is $\mathbb{Z}^2$}
{When the indexing set is Z2}}

Consider $M$ to be a persistence module over $\mathbb{R}$ with $\mathbb{Z}^2$ as the indexing set. Then there are real linear maps 
\[
f_{(i,j)}^{(i+r,j+s)} : M_{i,j} \to M_{i+r,j+s},
\]
which are compositions of horizontal maps from $M_{i,j}$ to $M_{i+1,j}$ and vertical maps from $M_{i,j}$ to $M_{i,j+1}$. We denote the linear maps from $M_{i,j}$ to $M_{i+r,j}$ by ${}_{j}H_i^{\,r}$ and the linear maps from $M_{i,j}$ to $M_{i,j+s}$ by ${}_{i}V_j^{\,s}$.

To ensure the freeness of a given persistence module, we need to add one criterion which is similar to injectivity in the context of the indexing set $\mathbb{Z}$. To simplify this, we obtain the following result.
\begin{lemma}\label{lem2d}
Assume that
\(
\operatorname{Im}({}_{j+1}H_{i}^{\,1})
\cap
\operatorname{Im}({}_{i+1}V_{j}^{\,1})
=
\operatorname{Im}\!\big(f^{(i+1,j+1)}_{(i,j)}\big),
\)
where $f^{(i+1,j+1)}_{(i,j)}={}_{i+1}V_j^1 \circ{}_jH_i^1={}_{j+1}H_i^1\circ {}_iV_j^1$.
Then for any $r,s\in\mathbb{N}$,
\(
\operatorname{Im}({}_{j+s}H_{i}^{\,r})
\cap
\operatorname{Im}({}_{i+r}V_{j}^{\,s})
=
\operatorname{Im}\!\big(f^{(i+r,j+s)}_{(i,j)}\big).
\)
\end{lemma}

\begin{proof}
It is enough to prove the case $r=2$, $s=1$, since the general case follows by similar diagram chasing.
\[
\begin{tikzcd}[row sep=3.5em, column sep=4em]
M_{i,j+1} 
  \arrow[r,"{}_{j+1}H_i^{\,1}"] 
&
M_{i+1,j+1} 
  \arrow[r,"{}_{j+1}H_{i+1}^{\,1}"] 
&
M_{i+2,j+1}
\\
M_{i,j}
  \arrow[r,"{}_{j}H_i^{\,1}"']
  \arrow[u,"{}_{i}V_j^{\,1}"'] 
  \arrow[ur, dashed, "f_{i,j}^{i+1,j+1}"']
&
M_{i+1,j}
  \arrow[r,"{}_{j}H_{i+1}^{\,1}"']
  \arrow[u,"{}_{i+1}V_j^{\,1}"'] 
  \arrow[ur, dashed, "f_{i+1,j}^{i+2,j+1}"']
&
M_{i+2,j}
\arrow[u,"{}_{i+2}V_j^{\,1}"']
\end{tikzcd}
\]
Let 
\(
x\in
\operatorname{Im}({}_{j+1}H_{i}^{\,2})
\cap
\operatorname{Im}({}_{i+2}V_{j}^{\,1}).
\)
Then $x={}_{j+1}H_{i}^{\,2}(y)$ for some $y\in M_{i,j+1}$.
Hence
\[
{}_{j+1}H_{i+1}^{\,1}\big({}_{j+1}H_{i}^{\,1}(y)\big)
\in
\operatorname{Im}({}_{j+1}H_{i+1}^{\,1})
\cap
\operatorname{Im}({}_{i+2}V_{j}^{\,1}).
\]
By the hypothesis and injectivity of ${}_{j+1}H_{i+1}^{\,1}$,
\(
{}_{j+1}H_{i}^{\,1}(y)
\in
\operatorname{Im}({}_{i+1}V_{j}^{\,1}).
\)
Again applying the hypothesis to the first square of the commutative diagram and using injectivity of ${}_{j+1}H_{i}^{\,1}$, there exists $z\in M_{i,j}$ such that
\(
y={}_{i}V_{j}^{\,1}(z).
\)
Therefore,
\[
x
=
{}_{j+1}H_{i}^{\,2}(y)
=
{}_{j+1}H_{i}^{\,2}\big({}_{i}V_{j}^{\,1}(z)\big)
=
f^{(i+2,j+1)}_{(i,j)}(z).
\]
Hence
\(
x\in \operatorname{Im}\!\big(f^{(i+2,j+1)}_{(i,j)}\big),
\)
which proves the claim.
\end{proof}

We now state a set of sufficient conditions for a persistence module to be free. Since the indexing set is $\mathbb{Z}^2$, according to the terminology of $\S$~\ref{sub23} we have $U_0=\mathbb{Z}_{\ge 0}\times\mathbb{Z}_{\ge 0}$. Hence $M$ is a graded $\mathbb{R}[x,y]$--module.

\begin{proposition}\label{propalg2d}
Let $\alpha<\beta$ and $\gamma<\delta$.  
Suppose \(
A=\{(i,j)\mid i<\alpha\},\;
B=\{(i,j)\mid i>\beta\},\;
C=\{(i,j)\mid j<\gamma\},\;
D=\{(i,j)\mid j>\delta\}.
\) Let $M_{ij}=0$ for all $(i,j)\in A\cup C$ and $\dim_{\mathbb{R}}M_{ij}=d_{ij}<\infty$ for all $i,j.$ Let the maps
\(
{}_{j}H_{i}^{\,1}:M_{i,j}\to M_{i+1,j},
\;
{}_{i}V_{j}^{\,1}:M_{i,j}\to M_{i,j+1}
\)
be given by\\
\(
{}_{j}H_{i}^{\,1}=
\begin{cases}
0, & (i,j)\in A\cup C,\\[2mm]
\text{an isomorphism}, & (i,j)\in B\setminus C,\\[2mm]
{}_{\delta}H_{i}^{\,1}, & (i,j)\in D\setminus(A\cup B),
\end{cases}\) and \\\(
 {}_{i}V_{j}^{\,1}=
\begin{cases}
0, & (i,j)\in A\cup C,\\[2mm]
\text{an isomorphism}, & (i,j)\in D\setminus A,\\[2mm]
{}_{\beta}V_{j}^{\,1}, & (i,j)\in B\setminus(C\cup D).
\end{cases}
\)\\
Set
\(
R=\mathbb{Z}^2\setminus (A\cup B\cup C\cup D).
\)
On the rectangle $R$ all the maps are injective. Assume further that
\[
\operatorname{Im}({}_{j+1}H_{i}^{\,1})\cap
\operatorname{Im}({}_{i+1}V_{j}^{\,1})
=
\operatorname{Im}\!\big(f_{(i,j)}^{(i+1,j+1)}\big)
\]
for all $(i,j)\in R$.
Then $M$ is a graded free.
\end{proposition}
\begin{proof}
In each $M_{k,\ell}$ choose elements $e^{(k,\ell)}_j$, $j=1,\dots,b_{k,\ell}$, such that 
$\{\overline{e^{(k,\ell)}_j}\}$ form a basis of $M_{k,\ell}/D_{k,\ell}$.  
Note that $M_{k,\ell}/D_{k,\ell}=0$ for all $(k,\ell)\notin R$.
Let $z\in M_{r,s}$ with $r,s\ge 0$. Then
\(
z=\operatorname{Im}(x^{i}y^{j}(z'))
\)
for some $z'\in M_{r',s'}$, where $(r',s')$ lies on the boundary of $R$, since from $(r',s')$ to $(r,s)$ one can pass through isomorphisms. Hence it is enough to consider $z\in M_{r,s}$ with $(r,s)\in R$.
If $z\notin D_{r,s}$, choose $\{e^{(r,s)}_j\}$ such that
\[
z-\sum_{j=1}^{b_{r,s}} c^{(r,s)}_j e^{(r,s)}_j = z_1 \in D_{r,s}.
\]
Then
\(
z_1 = {}_{s}H^{\,1}_{r-1}(y_1) + {}_{r}V^{\,1}_{s-1}(y_2),
\)
where $|y_1|, |y_2| < |z|$.  
We repeat the same procedure for $y_1$ and $y_2$.  
Since the number of elements in $R$ is finite and $M_{k,\ell}=0$ for $(k,\ell)\in A\cup C$, the process terminates after finitely many steps.
It remains to prove uniqueness.

It is enough to prove that
\[
\sum_{j=1}^{b} c_{j}^{(k_j,l_j)} x^{\,r-k_j} y^{\,s-l_j} e_{j}^{(k_j,l_j)} = 0 
\tag{\#}
\]
implies all coefficients $c_{j}^{(k_j,l_j)}=0$.
Let 
\[
S:=\{(k_j,l_j)\;:\; \text{those indices which occur as indices of } 
e_{j}^{(k_j,l_j)} \text{ in (\#)}\}.
\]
Write 
\(
\max S=(M_1,M_2), 
\;
\min S=(m_1,m_2).
\)
Observe that $(M_1,M_2),(m_1,m_2)\in R$.
Suppose $(r,s)$ is either on the right of or above $R$. Then we can write (\#) as
\[
x^{\,r-M_1} y^{\,s-M_2}
\left(
\sum_{j=1}^{b} c_{j}^{(k_j,l_j)}
x^{\,M_1-k_j} y^{\,M_2-l_j}
e_{j}^{(k_j,l_j)}
\right)=0.
\]
By injectivity of the maps $H$ and $V$, this implies
\(
\sum_{j=1}^{b} c_{j}^{(k_j,l_j)}
x^{\,M_1-k_j} y^{\,M_2-l_j}
e_{j}^{(k_j,l_j)}=0.
\)
Hence it is sufficient to choose $(r,s)=(M_1,M_2)$.

\medskip
\noindent
\textbf{Case 1.}
There exists $a\in\{1,2,\dots,b\}$ such that
\(
(k_j,l_j)=(M_1,M_2)
\; \text{for } j=a+1,a+2,\dots,b,
\)
that is, $(M_1,M_2)$ occurs at a highest degree in the summation of (\#).
Then
\[
\sum_{j=a+1}^{b} c_{j}^{(k_j,l_j)} e_{j}^{(k_j,l_j)}
=
-\sum_{j=1}^{a} c_{j}^{(k_j,l_j)}
x^{\,M_1-k_j} y^{\,M_2-l_j}
e_{j}^{(k_j,l_j)}
\in D_{M_1,M_2}.
\]
This easily implies $c_{j}^{(k_j,l_j)}=0$ for $j=a+1,\dots,b$. What remains to prove, will be proved in Case 2.

\medskip
\noindent
\textbf{Case 2.}
Suppose $a=b$. For two elements the base case is
\(
c_{1}^{(k_1,l_1)} x^{\,M_1-k_1} y^{\,M_2-l_1} e_{1}^{(k_1,l_1)}
+
c_{2}^{(k_2,l_2)} x^{\,M_1-k_2} y^{\,M_2-l_2} e_{2}^{(k_2,l_2)}
=0.
\)
By injectivity of the maps and using Lemma~\ref{lem2d} for the commutative diagram
with corner points $M_{(m_1,m_2)}$, $M_{(k_1,l_1)}$, $M_{(k_2,l_2)}$ and
$M_{(M_1,M_2)}$, it follows that if $c_{1}^{(k_1,l_1)}\neq 0$, then
$e_{1}^{(k_1,l_1)}$ is decomposable, which is a contradiction. Hence
$c_{1}^{(k_1,l_1)}=0$.
It is enough to prove in the following case: some of the coefficients in (\#) do not contain either a positive power of $x$ or a positive power of $y$.  
We take the terms that contain positive powers of $x$ on the left side and the terms that do not contain a positive power of $x$ on the right side. Thus (\#) can be written as
\(
x^{i}(E)=y^{j}(F),
\)
where $F$ is indecomposable in the corresponding degree.  
Using Lemma~\ref{lem2d} for the commutative diagram with corner points
$M_{(M_1,M_2)}$, $M_{(M_1-i,M_2)}$, $M_{(M_1,M_2-j)}$, and $M_{(M_1-i,M_2-j)}$,
together with the injectivity of the maps, we obtain that $F$ lies in the image of
$M_{(M_1-i,M_2-j)}$, and hence is decomposable, which is a contradiction.

\end{proof}
First, we need to check whether the hypotheses of Proposition \ref{propalg2d} are satisfied. The hypothesis \(
\operatorname{Im}({}_{j+1}H_{i}^{\,1})\cap
\operatorname{Im}({}_{i+1}V_{j}^{\,1})
=
\operatorname{Im}\!\big(f_{(i,j)}^{(i+1,j+1)}\big)
\) is equivalent to the fact that $\dim(\operatorname{Im}({}_{j+1}H_{i}^{\,1})\cap
\operatorname{Im}({}_{i+1}V_{j}^{\,1}))=\dim(\operatorname{Im}({}_{j+1}H_{i}^{\,1}))+\dim(
\operatorname{Im}({}_{i+1}V_{j}^{\,1}))-\dim(\operatorname{Im}({}_{j+1}H_{i}^{\,1})+
\operatorname{Im}({}_{i+1}V_{j}^{\,1}))$ is equal to $d_{i,j}.$

\begin{algorithm}[H]\label{2d}
\caption{Find an $\mathbb{R}[x,y]$-basis of $M$}
\KwIn{Matrices $({}_{j}H_i^1)_{d_{i+1,j}\times d_{i,j}}$ and 
$({}_{i}V_j^1)_{d_{i,j+1}\times d_{i,j}}$ for $\alpha\le i\le\beta$ and $\gamma\le j\le\delta$}
\KwOut{An $\mathbb{R}[x,y]$-basis $B$}
Initialize $B\leftarrow$ columns of $I_{d_{\alpha,\gamma}}$\;
\For{$i=\alpha,\alpha+1,\dots,\beta$}{
\For{$j=\gamma,\gamma+1,\dots,\delta$}{
\If{${}_{i+1}V_j^1 \times{}_jH_i^1\neq{}_{j+1}H_i^1\times {}_iV_j^1$}{\textbf{stop algorithm }\;}
$\overline{{}_{j}H_i^1}\leftarrow\operatorname{RREF}({}_{j}H_i^1)$\;
$h_{i,j} \leftarrow$ number of nonzero rows of $\overline{{}_{j}H_i^1}$\;
\If{$h_{i,j}<d_{i,j}$}{\textbf{stop algorithm }\;}
$\overline{{}_{i}V_j^1}\leftarrow\operatorname{RREF}({}_{i}H_j^1)$\;
$v_{j,i} \leftarrow$ number of nonzero rows of $\overline{{}_{i}V_j^1}$\;
\If{$v_{j,i}<d_{i,j}$}{\textbf{stop algorithm }\;}
$R_{i,j}\mid E_{i,j}
\leftarrow
\operatorname{RREF}\big(({}_{j+1}H_i^1\mid{}_{i}V_{j+1}^1)\mid I_{d_{i+1,j+1}}\big)$\;

$r_{i,j}\leftarrow$ number of nonzero rows of $R_{i,j}$\;

\eIf{$r_{i,j}<d_{i+1,j}+d_{i,j+1}-d_{i,j}$}{
\textbf{stop algorithm}\;
}{$P_{i,j}\leftarrow$ truncate first $r_{i,j}$ columns from the left of $I_{d_{i+1,j+1}}$\;

$B\leftarrow B\cup\big(E_{i,j}^{-1}\times$ columns of $P_{i,j}\big)$\;}

}
}

\Return{$B$}
\end{algorithm}
We now compute the complexity of the above algorithm.
\begin{proposition}
Let
\(
N=(\beta-\alpha+1)(\delta-\gamma+1)
\)
denote the number of bi-graded indices and set
\(
D=\max_{i,j} d_{i,j}.
\)
Then Algorithm~\ref{2d} performs at most
\(
O(ND^{3})
\)
arithmetic operations over $\mathbb{R}$.
\end{proposition}

\begin{proof}
We estimate the cost of each step of the algorithm. Constructing the identity matrix
$I_{d_{\alpha,\gamma}}$
requires at most $O(D^{2})$ operations.
For the double loop, the indices $(i,j)$ range over at most
\(
N=(\beta-\alpha+1)(\delta-\gamma+1)
\)
pairs:\\
\textit{(i) Commutativity test.}
The condition
\(
{}_{i+1}V_j^1\,{}_jH_i^1
=
{}_{j+1}H_i^1\,{}_iV_j^1
\)
requires multiplication of matrices whose dimensions are bounded by $D$.
Each multiplication costs $O(D^{3})$ operations.\\
\textit{(ii) Rank verification.}
The matrices
${}_{j}H_i^1$ and ${}_{i}V_j^1$
have size at most $D\times D$.
Computing their row-reduced echelon forms therefore requires
$O(D^{3})$ operations each (see \cite{jacquard2023barcode}).\\
\textit{(iii) Kernel computation.}
The algorithm computes
\(
\operatorname{RREF}
\big(({}_{j+1}H_i^1 \mid {}_{i}V_{j+1}^1)
\mid I_{d_{i+1,j+1}}\big).
\)
The augmented matrix has at most
$D$ rows and $3D$ columns.
Gaussian elimination on such a matrix requires
$O(D^{3})$ arithmetic operations (see \cite{jacquard2023barcode}).\\
\textit{(iv) Remaining operations.}
Counting nonzero rows, truncating identity matrices,
matrix multiplications with $E_{i,j}^{-1}$,
and updating the set $B$
each require at most $O(D^{2})$ operations
and are dominated by the reduction step above.

Since each iteration of the double loop costs
$O(D^{3})$ operations and there are at most $N$
iterations, the total running time is
\(
O(ND^{3}).
\)
\end{proof}

\section*{Acknowledgements}
Giriraj Ghosh gratefully acknowledges the financial support provided by the University Grants Commission (UGC), Government of India, through a doctoral fellowship (NTA Ref. No. 231610064231). The authors thank Bodhayan Roy for his suggestions on the algorithm part of the paper.


\begin{thebibliography}{99}

\bibitem{blanchette2023exact}
B.~Blanchette, T.~Br\"ustle, and E.~J.~Hanson,
\newblock Exact structures for persistence modules,
\newblock \href{https://arxiv.org/html/2308.01790v3}
{arxiv:2308.01790v3}.

\bibitem{bubenik2014categorification}
P.~Bubenik and J.~A.~Scott,
\newblock Categorification of persistent homology,
\newblock {\em Discrete \& Computational Geometry}, 51(3):600--627, 2014.
\newblock \url{https://link.springer.com/article/10.1007/s00454-014-9573-x}

\bibitem{bubenik2015metrics}
P.~Bubenik, V.~de~Silva, and J.~Scott,
\newblock Metrics for generalized persistence modules,
\newblock {\em Foundations of Computational Mathematics}, 15(6):1501--1531, 2015.
\newblock \href{https://doi.org/10.1007/s10208-014-9229-5}{doi:10.1007/s10208-014-9229-5}.

\bibitem{bubenik2021homological}
P.~Bubenik and N.~Mili\`cevi\'c,
\newblock Homological algebra for persistence modules,
\newblock {\em Foundations of Computational Mathematics}, 21(5):1233--1278, 2021.
\newblock \href{https://link.springer.com/article/10.1007/s10208-020-09482-9}{doi:10.1007/s10208-020-09478-8}.

\bibitem{carlsson2009multidimensional}
G.~Carlsson and A.~Zomorodian,
\newblock The theory of multidimensional persistence,
\newblock {\em Discrete \& Computational Geometry}, 42(1):71--93, 2009.
\newblock \href{https://doi.org/10.1007/s00454-009-9176-0}{doi:10.1007/s00454-009-9176-0}.

\bibitem{chazal2016structure}
F.~Chazal, V.~de~Silva, M.~Glisse, and S.~Oudot,
\newblock The structure and stability of persistence modules,
\newblock {\em SpringerBriefs in Mathematics},
\newblock Springer, Cham, 2016.
\newblock \href{https://doi.org/10.1007/978-3-319-42545-0}{doi:10.1007/978-3-319-42545-0}.

\bibitem{crawleyboevey2015decomposition}
W.~Crawley-Boevey,
\newblock Decomposition of pointwise finite-dimensional persistence modules,
\newblock {\em Journal of Algebra and Its Applications}, 14(5):1550066, 2015.
\newblock \href{https://doi.org/10.1142/S0219498815500668}{doi:10.1142/S0219498815500668}.

\bibitem{curry2014sheaves}
J.~M.~Curry,
\newblock {\em Sheaves, Cosheaves and Applications},
\newblock PhD thesis, University of Pennsylvania, 2014.
\newblock \href{https://arxiv.org/abs/1303.3255}{arXiv:1303.3255}.

\bibitem{dey2022computational}
T.~K.~Dey and Y.~Wang,
\newblock {\em Computational Topology for Data Analysis},
\newblock Cambridge University Press, Cambridge, United Kingdom, 2022.
\newblock \href{https://doi.org/10.1017/9781009099950}{doi:10.1017/9781009099950}.

\bibitem{gabriel1972unzerlegbare}
P.~Gabriel,
\newblock Unzerlegbare Darstellungen I,
\newblock {\em Manuscripta Mathematica}, 6(1):71--103, 1972.
\newblock \href{https://doi.org/10.1007/BF01298413}{doi:10.1007/BF01298413}.

\bibitem{jacquard2023barcode}
E.~Jacquard, V.~Nanda, and U.~Tillmann,
\newblock The space of barcode bases for persistence modules,
\newblock {\em Journal of Applied and Computational Topology}, 7:1--30, 2023.
\newblock \href{https://doi.org/10.1007/s41468-022-00094-6}
{doi:10.1007/s41468-022-00094-6}.


\bibitem{lam1978serre}
T.~Y.~Lam,
\newblock {\em Serre's Conjecture},
\newblock Lecture Notes in Mathematics, vol.~635, Springer, Berlin, 1978.
\newblock \url{https://link.springer.com/book/10.1007/BFb0068340}

\bibitem{lesnick2015interleaving}
M.~Lesnick,
\newblock The theory of the interleaving distance on multidimensional persistence modules,
\newblock {\em Foundations of Computational Mathematics}, 15(3):613--650, 2015.
\newblock \href{https://doi.org/10.1007/s10208-015-9255-y}{doi:10.1007/s10208-015-9255-y}.

\bibitem{lesnick2019interactive}
M.~Lesnick and M.~Wright,
\newblock Interactive visualization of 2-d persistence modules,
\newblock \href{https://arxiv.org/abs/1512.00180}{arXiv:1512.00180.}.

\bibitem{miller2020posets}
E.~Miller,
\newblock Homological algebra of modules over posets,
\newblock {\em SIAM Journal on Applied Algebra and Geometry}, 4(2):399--466, 2020.
\newblock \href{https://epubs.siam.org/doi/10.1137/22M1516361}{doi:10.1137/19M1284080}.

\bibitem{nastasescu2004graded}
C.~N\u{a}st\u{a}sescu and F.~Van~Oystaeyen,
\newblock {\em Methods of Graded Rings},
\newblock Lecture Notes in Mathematics, Vol.~1836,
\newblock Springer, Berlin--Heidelberg, 2004.
\newblock \href{https://link.springer.com/book/10.1007/b94904}{SpringerLink}.

\bibitem{serre1955fac}
J.-P.~Serre,
\newblock Faisceaux alg\'ebriques coh\'erents,
\newblock {\em Annals of Mathematics}, 61(2):197--278, 1955.
\newblock \url{https://www.jstor.org/stable/1969915}

\bibitem{stacks-project}
The {Stacks Project Authors},
\newblock {\em The Stacks Project},
\newblock \href{https://stacks.math.columbia.edu}{https://stacks.math.columbia.edu}.

\bibitem{webb1985decomposition}
C.~Webb,
\newblock Decomposition of graded modules,
\newblock {\em Proceedings of the American Mathematical Society}, 94(4):565--571, 1985.
\newblock \href{https://www.semanticscholar.org/paper/Decomposition-of-graded-modules-Webb/74218297b6f4c00a56de7f4c700516cb5c61de66}{doi:10.1090/S0002-9939-1985-0786058-3}.

\bibitem{zomorodian2005computing}
A.~Zomorodian and G.~Carlsson,
\newblock Computing persistent homology,
\newblock {\em Discrete \& Computational Geometry}, 33(2):249--274, 2005.
\newblock \href{https://doi.org/10.1007/s00454-004-1146-y}{doi:10.1007/s00454-004-1146-y}.

\end{thebibliography}
\end{document}